\documentclass[10pt]{article}

\usepackage{color}
\usepackage{latexsym}
\usepackage{amssymb,amsmath,amstext}
\usepackage{epsfig}
\usepackage{graphics}
\usepackage{subfigure}
\usepackage{euscript}
\usepackage{ifthen}
\usepackage{wrapfig}
\usepackage{amsthm}

\newtheorem{theorem}{Theorem}

\newtheorem{corollary}{Corollary}
\newtheorem{lemma}{Lemma}
\newtheorem{definition}{Definition}

\newtheorem{remark}{Remark}

\newtheorem{algorithm}{Algorithm}

\newcounter{algo_line}

\begin{document}
{

\title{The rank of a divisor on a finite graph: geometry and computation}

\author{Madhusudan Manjunath \\Universit\"at des Saarlandes and Max-Planck-Institut,\\ Saarbr\"ucken, Germany
}
\maketitle

{\bf Abstract:} We study the problem of computing the rank of a divisor on a finite graph, a quantity that arises in the Riemann-Roch theory on a finite graph developed by 
Baker and Norine ({\it Advances in Mathematics}, 215({\bf 2}): 766-788, 2007). Our work consists of two parts: 
 the first part is an algorithm whose running time is polynomial for a multigraph with a fixed number of vertices. More precisely, our algorithm has running time 
$O(2^{n \log n})\text{poly}(\text{size})(G)$, where $n+1$ is the number of vertices of the graph $G$. The second part consists of a 
new proof of the fact that testing if rank of a divisor is non-negative or not is in the complexity class $NP \cap co-NP$ and 
motivated by this proof and its generalisations, we construct a new graph invariant that we call the critical automorphism group of the graph.

\section{Introduction}

%%%%%%%%%%Start with the introduction to the theory of chip firing games and sandpile models; go onto to Baker and Norine, the history.

The study of  chip firing games on a graph with their several variants is a classic topic 
in discrete mathematics, in particular in algebraic graph theory and this is reflected in the fact 
that standard text books on algebraic graph theory, such as the book by Godsil and Royle \cite{GodRoy01}
have chapters devoted to chip firing games. See the article of Merino \cite{Mer05} for a broad survey of the topic.

%%%%%%%%Give a reference to the survery article
%%%%%%%%%% Recently, chip firing games have also appeatred in connection analgoies between graphs and algebraic curves and the closely related topic of tropical geometry.
%%%%%%%%%%%In this context, Baker and Norine used to the chip firing games to obtain an analgoue of Riemann-Roch theorem on graphs explain the Riemann-Roch.

%%%%%%%%%Define chip firing games and mention their connections to physics, elliptic curves.
Despite their simplicity, chip firing games have connections to various other areas of mathematics and physics. 
Some examples include chip firing games studied in  dynamical systems under the name ``sandpile models'' and 
in fact, early algorithmic work on chip firing games was done independently by physicists. 
Chip firing games of graphs also play a role in counting the number of points on elliptic curves over 
finite fields and is the topic of the Phd thesis of Greg Musikar.
Recently, chip firing games have played a key role in developing analogies between graphs and  
Riemann surfaces. This line of research was initiated by Baker and Norine in their pioneering work 
``Riemann-Roch and Abel-Jacobi theory on a finite graph'' \cite{BaNo07} where they show analogues of
the Riemann-Roch theorem on a finite graph and the theorem is best explained in the language of chip-firing games.

A chip-firing game is a solitary game played on an undirected connected multigraph and is defined as follows: 
Each vertex of the graph is assigned an integer, refereed to as ``chips'' and this assignment is called the initial configuration. 
At each move of the game, an arbitrary vertex $v$ is allowed to either lend or borrow one chip along each edge incident with it and we obtain a new configuration. 
We define two configurations $C_1$ and $C_2$ to be  equivalent if $C_1$ can be reached from $C_2$ by a sequence of chip firings. 
The Laplacian matrix $Q$ of the graph naturally comes into the picture as follows:

\begin{lemma}
Configurations $C_1$ and $C_2$ are equivalent if and only if $C_1-C_2$ can be expressed as $Q\cdot w$ for some vector $w$ with integer coordinates.
\end{lemma}

We can ask some natural questions on such a game:

\begin{enumerate}
\item Is a given configuration equivalent to an effective configuration i.e., a configuration where each vertex has a non-negative number of chips?

\item More generally, given a configuration what is the minimum number of chips that must be removed from the system so that 
the resulting configuration is not equivalent to an effective configuration?
\end{enumerate}

The Riemann-Roch theorem of Baker and Norine provides insights into answering these questions. 
We need the following definitions before we can state the theorem. 

\begin{definition}({\bf Divisor})
 A configuration on a finite graph is called a divisor and is represented as an integer vector on $n+1$ coordinates.
\end{definition}

\begin{definition}({\bf Degree of a divisor})
 For a divisor $D$, the total number of chips, i.e., the sum of chips over all the vertices of the graph, 
is called the degree of the divisor $D$, denoted by $deg(D)$.
\end{definition}

\begin{definition}({\bf Rank of a divisor})\label{rank_def}
For a divisor $D$, one less than the minimum number of chips that must be removed from the divisor $D$ so that the resulting configuration is not equivalent to an 
effective divisor is called the rank of the divisor $D$, denoted by $r(D)$.
\end{definition}

\begin{theorem}({\bf Riemann-Roch theorem for graphs})\label{RieRo_theo}
For any undirected connected graph $G$, there exists a divisor $K$ called the canonical divisor such that for any divisor $D$
we have the following formula:
\begin{equation}\label{RieRo_form}
r(D)-r(K-D)=deg(D)-(g-1)
\end{equation}
where $g$ is the cyclotomic number of $G$ and is equal to $m-n+1$ where $m$ is the number of edges, $n$ is the number of vertices of $G$.
\end{theorem}

%%%%%%%%%%%

\subsection{Related work and a brief description of our work}

The central quantity in the Riemann-Roch theorem is the rank of a divisor and the efficient computation of rank is a natural problem, attributed to 
Hendrik Lenstra (see \cite{HlaKraNor10}). In fact this work, gave an algorithm, i.e., a procedure that terminates in a finite number of steps to compute the rank of a divisor 
on a tropical curve. But the algorithm does not run in polynomial time in the size of the multigraph even when the number of vertices are fixed since the algorithm involves iterating over all the spanning trees in the graph (see Proof of Theorem 23 
in \cite{HlaKraNor10}) and the number of spanning trees in indeed not polynomially bounded in the size of the mutligraph even if the number of vertices are fixed. 
On the other hand, there are polynomial time algorithms for deciding if the rank of a divisor on a finite multigraph is non-negative, see \cite{Dha93},~\cite{Tar88} and \cite{Sho09}.

Our paper is centered around the problem of computing the rank. More precisely, (Section \ref{geominter_theo}) we obtain an algorithm whose running time is polynomial for a multigraph with a fixed number of vertices. 
More precisely, our algorithm has running time $2^{O(n \log n)}\text{poly}(\text{size})(G)$, where $n+1$ is the number of vertices of the multigraph $G$.
Recall that we are working with arbitrary undirected connected multigraphs or equivalently graphs with positive integer weights on the edges 
and indeed, the original Riemann-Roch theory was also developed in this setting. The main tools involved are the Riemann-Roch formula and a formula for rank 
(Theorem \ref{rankform_theo}) that is used in the proof of the Riemann-Roch formula, these results were first obtained in the work of 
Baker and Norine \cite{BaNo07}. We obtain a geometric interpretation of rank (Theorem \ref{geominter_theo}) and
combine this geometric interpretation along with algorithms from the geometry of numbers to 
obtain the algorithm for computing the rank (Algorithm \ref{mainalgo}).  We find it satisfying that geometric tools seem to be essential in obtaining the algorithm, 
though the definition of rank of a divisor can be stated in purely combinatorial terms.

%The results is in a similar flavour to  algorithms for integer programming that run in polynomial time for a fixed number of dimensions, 
%%for example the algorithms of Lenstra and Kannan. In fact, our algorithm uses the Kannan's algorithm for integer programming.

The second part (Section \ref{dual_sect}) starts with a new proof of the fact that testing if $r(D) \geq 0$ is in $NP \cap co-NP$. 
A duality theorem  characterising divisors with $r(D) \geq 0$ plays a key role in our proof. Motivated by the observation that generalisations of the duality theorem 
lead to more general complexity results on computing the rank,  we generalise the duality theorem and this generalisation leads to the construction of a new graph invariant that 
we call the {\bf critical automorphism group} of the graph.

\section{Preliminaries}

We will now describe mainly geometric notions and results that we frequently use in the rest of our paper.

\subsection{Lattices} \label{Lat_Sect}

A  {\bf lattice} $L$ is a discrete subgroup of the Euclidean vector space $\mathbb{R}^{n}$. More concretely,  a lattice is the Abelian group obtained by taking all the integral combinations of a set of linearly independent vectors $b_1, \dots,b_k$ in $\mathbb{R}^{n}$. More precisely,
 
\begin{equation}L=\{ \sum_{i=1}^{k} \alpha_i b_i|~\alpha_i \in \mathbb{Z}\} \end{equation}

The set $\mathcal{B}=\{ b_1, \dots,b_k \}$ is called a basis of $L$ and the integer $k$ is called the dimension of $L$, denoted by $dim(L)$, is independent of the choice of the basis. We denote the subspace spanned by the elements of $\mathcal{B}$ by $Span(L)$.

Let us now look at the important geometric invariants of a lattice. The volume of the lattice $Vol(L)$, also known as the discriminant or determinant, is defined as $\sqrt{det(\mathcal{B}\mathcal{B}^{t})}$ where the basis $\mathcal{B}$ is represented as a matrix with its elements row-wise and hence, $\mathcal{B}\mathcal{B}^{t}$ is the Gram matrix of the basis elements. 
Another important invariant is the norm of a shortest vector of a lattice.

\begin{definition}
A shortest vector of $L$ in the Euclidean norm is an element $q$ of $L$ such that $q \cdot q \leq q' \cdot q'$
for all non-zero elements $q'$ in $L$ and we denote $||q||_2$ as $\nu_E(L)$. 
\end{definition}

%%Two other important invariants of a lattice are the packing radius and the covering radius.

%%\begin{definition}
%%Let $B(q,R)$ be the Euclidean ball centered at the point $q \in \mathbb{R}$ and with radius $R$.
%%The packing radius Pac$_E(L)$ and the covering radius Cov$_E(L)$ of a lattice $L$ is defined as: 

%%Pac$_E(L)= \sup \{R|~B(q_1,R) \cap B(q_2,R)=\emptyset~\forall q_1,~q_2 \in L,~q_1 \neq q_2\}$,

%%Cov$_E(L)= \inf \{R|~\forall p \in Span(L)~\exists q \in L: p \in B(q,R) \}$.

%%\end{definition}
%%The shortest vector and the packing radius are related by Pac$_E(L)=\nu_E(L)/2$. The study of relations between geometric invariants of a lattice is a classical 
%%topic \cite{Siegel89}. To convey the flavour of the topic, we state a basic result known as Minkowski's First Theorem:

%%\begin{theorem}{\bf (Minkowski's First Theorem)}
%%A lattice $L$ has a non-zero element whose Euclidean norm is upper bounded by $2Vol^{1/n}(L)/V_E^{1/n}$
%%where $V_E$ is the volume of the unit ball in $Span(L)$.
%%\end{theorem}
%%\begin{proof}
%%See Lecture II, Section 5 of Siegel \cite{Siegel89}.
%%\end{proof}

\subsection{The Laplacian Lattice of a Graph}

%%%Define Laplacian lattices and state the very basic results.
%% \subsection{The Laplacian matrix of a Graph}
For an undirected connected graph $G$, the Laplacian matrix $Q(G)$ is defined as $D(G)-A(G)$ where $D(G)$
is the diagonal matrix with the degree of every vertex in its diagonal and $A(G)$ is the vertex-adjacency matrix of the graph. 
We assume the following standard form of the Laplacian matrix:

\begin{equation}\label{gra_form}
Q=
\begin{bmatrix}
  \delta_0      & -b_{01} & -b_{02}  \hdots & -b_{0n} \\
 -b_{10}   &  \delta_1    & -b_{12}  \hdots & -b_{1n} \\
 \vdots    &  \vdots & \ddots\\
 -b_{n0}   &  -b_{n1} & -b_{n2}  \hdots &  \delta_n
\end{bmatrix}
\end{equation}
has the following properties:
\begin{itemize}
\item[$(C_1)$] $b_{ij}$'s are integers, $b_{ij}\geq0$ for all $0\leq i \neq j\leq n$ and $b_{ij}=b_{ji},\:\:\forall i\neq j$.
\item[$(C_2)$] $\delta_i=\sum_{j=1,j\neq i}^{n}b_{ij}=\sum_{j=1,j\neq i}^{n}b_{ji}$ (and is the degree of the $i$-th vertex).
% \item[$(C_3)$] If $G$ is a graph,  and $\delta_i$ .
\end{itemize}
%  To simplify the presentation we restrict to Laplacian of multi-graphs and refer to the full version of this paper for the general case. 
%We denote by $B$ the basis $\{b_0,\dots,b_{n-1}\}$ of $L_G$.

Though the Laplacian matrix of a graph contains essentially the same information as the adjacency matrix of a graph, it enjoys other nice properties. See, for example, the chapter ``The Laplacian matrix of a Graph'' in the algebraic graph theory book of Godsil and Royle \cite{GodRoy01} for a more complete discussion.

\begin{lemma}The Laplacian matrix $Q(G)$ is a symmetric positive semi-definite matrix.\end{lemma}

Another remarkable property of the Laplacian matrix is described in the Matrix-Tree theorem:

\begin{theorem} {\bf (Kirchoff's Matrix-Tree Theorem)}The absolute of value of any cofactor is equal to the number of spanning trees of the graph.\end{theorem}

A remarkable aspect of the matrix-tree theorem is that it reduces counting the number of spanning trees of 
a graph into a determinant computation and hence provides a polynomial time algorithm for it.

\begin{definition}{\bf (The Laplacian lattice of a graph)}
Given the Laplacian matrix $Q(G)$, the lattice generated by the rows (or equivalently the columns) of $Q(G)$ is called  the Laplacian lattice $L_G$ of the graph. 
\end{definition}

In the language of chip firing games, the Laplacian lattice is the set of all divisors that are equivalent to the divisor $(0,\dots,0)$.

\begin{definition}{\bf (The Hyperplanes $H_k$)}
For a fixed real number $k$,  we denote $n$-dimensional hyperplane  $\{(x_0,\dots,x_n)|~ \sum_{i=0}^{n}x_i=k\}$ by $H_k$. 
\end{definition}

\begin{definition}{\bf (The Root Lattice $A_n$)}
The root lattice $A_n$ is the lattice of integer points in the hyperplane $H_0=\{(x_0,\dots,x_n)|~ \sum_{i=0}^{n+1}x_i=0,~x_i \in \mathbb{R}\}=(1,\dots,1)^{\perp}$.
More precisely, 
\begin{equation}\notag A_n=\{(x_0,\dots,x_n)|~ \sum_{i=0}^{n+1}x_i=0,~x_i \in \mathbb{Z}\}.\end{equation}
\end{definition}

\begin{remark}
The name ``root lattice'' is derived from the fact that the lattice $A_n$ is generated by a root system i.e. 
a set of vectors that satisfy reflection symmetries. See pages 96--98 of Conway and Sloane  \cite{ConSlo99}
for a precise definition of a root system and for a discussion on root lattices. In the case of 
$A_n$, the corresponding root system is $e_i-e_j$ where $e_i$ and $e_j$ run over the standard basis of $\mathbb{R}^{n+1}$.
\end{remark}

We now make a few simple observations on the Laplacian lattice of a graph.

\begin{lemma} The Laplacian lattice of a graph on $n+1$-vertices is a sublattice of the root lattice $A_n$.\end{lemma}

\begin{definition}{\bf (The covolume of a sublattice)}
A full-dimensional sublattice $L_s$ of a lattice $L$ is a subgroup of the Abelian group $L$ 
the cardinality of the quotient group $L/L_s$ is called the covolume of $L_s$ with respect to $L$. \end{definition}

\begin{lemma}\label{covol_lem}\cite{Sho09} The covolume of the Laplacian lattice of $G$ with respect to $A_n$ is equal 
to the number of spanning trees of $G$.\end{lemma}
%%\begin{proof}
%%{\color{} Add a proof.}
%%\end{proof}

The elements of $A_n/L_G$ naturally possess an Abelian group structure and this group is known as the Picard group of $G$, also known as the Jacobian of $G$. A number of works have been devoted to the study of the structure of this group and 
the information that it contains about the underlying graph, see for example the works of Biggs \cite{Big97}, Kotani and Sunada \cite{KotSun98} and Lorenzini \cite{Lor08}. 
As a straightforward corollary to Lemma \ref{covol_lem} we obtain:

\begin{corollary}
The cardinality of the Picard group of $G$ is equal to the number of spanning trees of $G$.
\end{corollary}

%%See \cite{Sho09} for a new ``bijective'' proof of the Matrix-Tree theorem and for an application 
%%of the theorem to random sampling of spanning trees. 

%%%%%%%%%We can mention the lattice of integral flows and cuts and the relation with the Laplacian lattice.

\subsection{Algorithmic Geometry of Numbers}

%%%%%%%%%%%%%%%Mention the main algorithmic problems in geometry of numbers: the geometric ones are  CVP, SVP and integer programming 
%%%%%%%%%%%%%%their relations, the algebraic one: basis equivalence.
We will now briefly discuss some important algorithmic problems related to lattices and geometry of numbers, in particular 
those that we employ in our algorithm. Two important and closely related problems in the algorithmic geometry of numbers are:

\begin{itemize}

\item {\bf Closest vector problem (CVP)}: Given a lattice $L$ by an arbitrary basis and a target vector $v$, find a lattice that is closer to $v$ than to any other lattice point in the $\ell_2$-norm.

\item {\bf Shortest vector problem (SVP)}: Given a lattice $L$, by an arbitrary basis find a non-zero lattice point that has the smallest $\ell_2$-norm. 
\end{itemize}

Indeed, CVP and SVP have versions with respect to the other $\ell_p$-norms. For the $\ell_2$-norm both CVP and SVP are 
known to be NP-hard and in fact SVP has a polynomial time reduction to CVP in any given norm but the converse is not known.
A  problem that is closely related to CVP is the integer programming problem:

 {\bf Integer programming problem}: Given a polyhedron $\mathcal{P}$ in the form $A \cdot x \leq b$ where $A$ is an $m \times n$ matrix and $b$ is a vector in $\mathbb{R}^m$.
 Decide if $\mathcal{P}$ has an integer point or not i.e., $\mathcal{P} \cap \mathbb{Z}^n = \emptyset$ or not.

%%This algorithm along with the celebrated notion
%%of LLL-reduced basis led to a large number of applications of lattice reduction, see the book \cite{LLL10} that is dedicated to the LLL-reduced basis and its applications. 

Indeed, the integer programming problem is also known to be NP-hard. In the 1980's there was a great amount of progress in algorithmic geometry of numbers, trigerred by
 the algorithm of Lenstra that solves the integer programming problem in polynomial time for a fixed dimension\cite{Len83}. 
Lenstra's algorithm has running time  $2^{O(n^3))}\cdot poly(|I|)$, where $|I|$ is the size of the input. The factor $2^{O(n^3)}$ was subsequently improved 
and the current best factor being $2^{O(n \log n)}$ by Kannan \cite{Kan87}. Note that the Kannan's algorithm also works when the polytope is presented as a separation oracle
and in fact, the algorithm works also for general convex bodies presented in terms of a separation oracle, see the remark ``General convex bodies and mixed integer programs'' in \cite{Kan87}. 
We will crucially use Kannan's algorithm and the polytope will be presented to Kannan's algorithm in terms of a separation oracle.

%%%%%%%%%%%%%%%Give reference to Kannan's survery and the more recent LLL; algorithm survey.

%%%%%%%%%%%%%%%Go on to mention the work on Kannan. Miccancio and the recent work of Vempala et al.

\subsection{Polyhedral Distance Functions}\label{polydistfunc_sect}

%%%Define the distance function induced by a convex body and in particular distance functions and prove the basic properties.
 Let $\mathcal{P}$ be a convex polytope in $\mathbb R^n$ with the reference point $O=(0,\dots,0)$ that we call the ``center'' in its interior. By $\mathcal{P}(p,\lambda)$ we denote a dilation of $\mathcal{P}$ by a factor $\lambda$ and its center translated to the point $p$ i.e. $\mathcal{P}(p,\lambda)=p+\lambda.\mathcal{P}$ and $\lambda.\mathcal{P}\:=\:\{\:\lambda.x\:|\:x\in \mathcal{P}\:\}$. 
We define the $\mathcal{P}$-midpoint of two points $p$ and $q$ in $\mathbb{R}^{n}$ as $\inf\{R|~\mathcal{P}(p,R) \cap \mathcal{P}(q,R) \neq \emptyset\}$. 
The {\it polyhedral distance function} $d_{\mathcal{P}}(.\:,.)$ between the points of $\mathbb R^n$ is defined as follows:
\[\forall\: p,q\in \mathbb R^n,\: d_{\mathcal{P}}(p,q)\::=\:\inf\{\lambda\geq 0\:|\:q \in \mathcal{P}(p,\lambda)\}.\]
$d_{\mathcal{P}}$ is not generally symmetric, indeed it is easy to check that $d_{\mathcal{P}}(.\:,.)$ is symmetric if and 
only if the polyhedron $\mathcal{P}$ is centrally symmetric i.e. $\mathcal{P}=-\mathcal{P}$. Nevertheless $d_{\mathcal{P}}(.\:,.)$ satisfies the triangle inequality. 

\begin{lemma}\label{lin_lem}
For every three points $p,q,r\in \mathbb R^{n}$, we have $d_{\mathcal{P}}(p,q)+d_{\mathcal{P}}(q,r) \geq d_{\mathcal{P}}(p,r)$.
 In addition, if $q$ is a convex combination of $p$ and $r$, then $d_{\mathcal{P}}(p,q)+d_{\mathcal{P}}(q,r)=d_{\mathcal{P}}(p,r)$.
\end{lemma}

\begin{proof}
To prove the triangle inequality, it will be sufficient to show that if $q\in p+\lambda.\mathcal{P}$ and $r\in q+\mu.\mathcal{P}$, then $r\in p+(\lambda+\mu).\mathcal{P}$. We write $q=p+\lambda.q'$ and $r=q+\mu.r'$ for two points $q'$ and $r'$ in $\mathcal{P}$. We can then write $r=p+\lambda.q'+\mu.r'=p+(\lambda+\mu)(\frac{\lambda}{\lambda+\mu}.q'+\frac{\mu}{\lambda+\mu}.r')$. $\mathcal{P}$ being convex and $\lambda,\mu\geq0$, we infer that $\frac{\lambda}{\lambda+\mu}.q'+\frac{\mu}{\lambda+\mu}.r'\in \mathcal{P}$, and so $r\in p+(\lambda+\mu).\mathcal{P}$. The triangle inequality follows.

\noindent To prove the second part of the lemma, let $t \in [0,1]$ be such that $q=t.p+(1-t).r\:.$ By the triangle inequality, it will be enough to prove that $d_{\mathcal{P}}(p,q)+d_{\mathcal{P}}(q,r)\leq d_{\mathcal{P}}(p,r)$. Let $d_{\mathcal{P}}(p,r)=\lambda$ so that $r = p+\lambda.r'$ for some point $r'$ in $\mathcal{P}$. We infer first that 
$q=t.p+(1-t).r=t.p+(1-t)(p+\lambda.r')= p+(1-t)\lambda.r'$, which implies that $d_{\mathcal{P}}(p,q)\leq(1-t)\lambda$. Similarly we have $t.r=t.p+t\lambda.r'=q-(1-t)r+t\lambda.r'$. It follows that $r=q+t\lambda r'$ and so $d_{\mathcal{P}}(q,r)\leq t\lambda\:.$ We conclude that $d_{\mathcal{P}}(p,q)+d_{\mathcal{P}}(q,r)\leq d_{\mathcal{P}}(p,r)$, and the lemma follows.
\end{proof}

We also observe that the polyhedral metric $d_{\mathcal{P}}(.\:,.)$ is translation invariant, i.e.
\begin{lemma}\label{lem:trans-inv} For any two points $p,q$ in $\mathbb R^n$, and for any vector $v\in\mathbb R^n$, we have $d_{\mathcal{P}}(p,q) = d_{\mathcal{P}}(p-v,q-v)$. In particular, $d_{\mathcal{P}}(p,q)=d_{\mathcal{P}}(p-q,O)=d_{\mathcal{P}}(O,q-p)$. 
\end{lemma}
\begin{proof} The proof is easy: if $q\in p+\lambda.\mathcal{P}$, then $q-v \in p-v+\lambda.\mathcal{P}$, and vice versa. 
\end{proof}

\begin{remark}\rm
The notion of a polyhedral distance function is essentially the concept of a gauge function of a convex body that has been studied in \cite{Siegel89}. 
Lemmas \ref{lin_lem} and \ref{lem:trans-inv} can be derived in a straight forward way from the results in \cite{Siegel89}. 
\end{remark}

Recall that the $n$-dimensional hyperplane $H_k$ is defined as $H_k=\{(x_0,\dots,x_n)|~ \sum_{i=0}^{n}x_i=k\}$. 
We will be mainly be using distance functions defined by regular simplices $\triangle$ and $\bar{\triangle}$ where the simplices $\triangle$ and $\bar{\triangle}$ are defined as follows: 

\begin{definition}
The regular simplex $\triangle$ is the convex hull of $t_0,\dots,t_{n} \in H_0$ where 
\label{deltaver_eq}\begin{equation}\notag
t_{ij}=
\begin{cases}
~~n, \text{~if}~i=j,\\
-1, \text{~otherwise}
\end{cases}
\end{equation}
 for $i$ from $0,\dots,n$ and $t_{ij}$ is the $j$-th coordinate of $t_i$. We define $\bar{\triangle}$ as $-\triangle$.
\end{definition}

We note that for points in $H_0$, the distance functions $d_{\triangle}$ and $d_{\bar{\triangle}}$ have a simple formula:

\begin{lemma} (Lemma 4.7, \cite{AmiMan10})
For any pair of points $p,~q$ in $H_0$, we have:
              \begin{gather}
           \notag d_{\triangle}(p,q)=|\min_i (q_i-p_i)|,\\
          \notag d_{\bar{\triangle}}(p,q)=|\min_i (p_i-q_i)|.
         \end{gather}
\end{lemma}

Given a permutation $\pi$ on the $n+1$ vertices of $G$, define the ordering $\pi(v_0) <_{\pi} \pi(v_1)<_{\pi} \dots <_{\pi}\pi(v_n)$ and 
orient the edges of graph $G$ according to the ordering defined by $\pi$ i.e., there is an oriented edge from $v_i$ to $v_j$ if $(v_i,v_j) \in E$ 
and if $v_i<_{\pi}v_j$ in the ordering defined by $\pi$. Consider the acyclic orientation induced by a permutation $\pi$ on the set of vertices of $G$ and
 define $\nu_{\pi}=(indeg_{\pi}(v_0)-1,\dots,indeg_{\pi}(v_{n})-1)$, where $indeg_{\pi}(v)$ is the indegree of the vertex $v$ in the directed graph oriented according to $\pi$. 
Define 
\begin{equation}
\text{Ext}(L_G)= \{-\nu_{\pi}+q|~ \pi \in S_{n+1},~ q \in L_G\}.
\end{equation}

%%\subsection{Geometric Invariants of a Lattice with respect to Polyhedral Distance Functions}

%%In Section \ref{Lat_Sect}, we defined some important geometric invariants of a lattice with respect to the Euclidean norm.
%%We define analogous invariants with respect to polyhedral distance functions.

%%\begin{definition}
%%An element $q$ of $L$ is called a shortest vector with respect to the polyhedral distance function $d_{\mathcal{P}}$ if
%%$d_{\mathcal{P}}(O,q) \leq d_{\mathcal{P}}(O,q')$ for all $q' \in L/\{O\}$, where $O$ is the origin.  We denote $d_{\mathcal{P}}(O,q)$ by $\nu_{\mathcal{P}}(L)$.
%%\end{definition}

%%Note that the shortest vector with respect the distance functions $d_{\mathcal{P}}$
%%and $d_{\bar{\mathcal{P}}}$ can be potentially different.

%%\begin{definition}
%%%For a lattice $L$, we define packing and covering radius of $L$ with respect $\mathcal{P}$  as:

%%Pac$_{\mathcal{P}}(L)= \sup \{R|~\mathcal{P}(q_1,R) \cap \mathcal{P}(q_2,R)=\emptyset,~ \forall q_1,~q_2 \in L ~q_1 \neq q_2\}$

%%Cov$_{\mathcal{P}}(L)= \inf \{R|~\forall p \in Span(L) \text{ is contained in } \mathcal{P}(q,R) \text{ for some } q \in L \}$

%%%\end{definition}

%%%%%%%%%%Interesting, though the problem of computing the rank can be posed as a purely combinatorial problem, our works relies heavily on geometric ideas and tools.
%%%%%%%%%This section is devoted to developing the geometric machinery that we frequently use in the paper.

%%%%%%%%%Introduce lattices and distance functions: some of their basic properties.

\subsection{Distance function induced by a Discrete Point Set}\label{distfunc_subsect}

Given a polyhedral distance function $\mathcal{P}$ and a discrete point set $S$, we define a function $h_{\mathcal{P},S}:\mathbb{R}^{n} \rightarrow \mathbb{R}$
as:
\begin{equation} h_{\mathcal{P},S}(p)= \min_{q \in S}\{d_{\mathcal{P}}(p,q)\}\end{equation}

In particular, the notion of local minima and local maxima of the distance function $h_{\mathcal{P},S}$ turns out to be useful:

\begin{definition}{\bf{(Local Maxima and Local Minima of $h_{\mathcal{P},S}$)}}
Let  $B(p,\epsilon)$ be the Euclidean ball of radius $\epsilon$ centered at $p$.
A point $c$ in $\mathbb{R}^{n}$ is called a local minimum of $h_{\mathcal{P},S}$ if there exists an $\epsilon>0$ such that
  $h_{\mathcal{P},S}(c) \leq h_{\mathcal{P},S}(q)$ for all $q \in B(c,\epsilon)$. A point $c$ in $\mathbb{R}^{n}$ is called a local maximum of $h_{\mathcal{P},S}$ if there exists an $\epsilon>0$ such that $h_{\mathcal{P},S}(c) \geq h_{\mathcal{P},S}(q)$ for all $q \in B(c,\epsilon)$.
\end{definition}

%%We have the following characterization of local minima of $h_{\mathcal{P},S}$:

%%\begin{lemma} A point $q \in \mathbb{R}^{n}$ is a local minimum of $h_{\mathcal{P},S}$ if and only if $q \in S$.\end{lemma}

We denote the set of local maxima of $h_{\mathcal{P},S}$ by Crit$_{\mathcal{P}}(S)$.

\section{Algorithms for computing the rank}\label{geomalg_sect}

In this section, we will construct algorithms for computing the rank with the main result being an algorithm for computing the rank that runs
 in polynomial time when the number of vertices of the multigraph is fixed. 

%%\subsection{The rank of a divisor and the Riemann-Roch theorem}

%%For a divisor $D$, the rank of $D$ is one less than the minimum number of chips we must remove from $D$ so that 
%%the resulting divisor is not effective. More precisely, 
%%  \begin{equation}r(D)=min \{\deg(E)|~|D-E|=\emptyset,~E \geq O\}-1\end{equation}

%%\begin{theorem} Let $G$ be an undirected connected graph, for every divisor $D$ we have:
%%  \begin{equation} r(D)-r(K-D)=deg(D)-g+1\end{equation}
%%where $K=(deg(v_1)-2,\dots,deg(v_{n+1})-2)$.
%% \end{theorem}

\subsection{A simplification}\label{simp_subsect}

We shall first observe that by using the Riemann-Roch theorem we can restrict our attention to divisors of degree between zero
and $g-1$. Firstly, a divisor of negative degree must have rank minus one. Furthermore, by the Riemann-Roch formula we have:

\begin{lemma} If the degree of $D$ is strictly greater than $2g-2$, then $r(D)=deg(D)-g$. \end{lemma}
\begin{proof} Observe that if the degree of $D$ is strictly greater than $2g-2$ then the rank of $K-D$ is $-1$ and apply the Riemann-Roch theorem. 
\end{proof}

Furthermore, we can compute the rank of divisors of degree between $g$ and $2g-2$ by computing the rank of $K-D$, a divisor that has degree 
between zero and $g-1$ and then applying the Riemann-Roch theorem. Hence, we consider the problem of computing the rank of a divisor of degree between zero and $g-1$. 
In fact, we consider the decision version of the problem i.e., we want to decide (efficiently) if $r(D) \leq k$ for every $k$ between zero and $g-1$; 
observe that such a procedure combined with a binary search over the parameter $k$ will compute the rank in time $O(\ln(g))$ times the running time of the procedure.

\subsection{A first attempt at computing the rank}

Let us discuss a first attempt at computing the rank. We will compute rank directly from its definition (Definition \ref{rank_def}). 
We will use the fact that there is a polynomial time algorithm for testing if $r(D) \geq 0$ 
due to the independent work of  Dhar \cite{Dha93} and Tardos \cite{Tar88}.

\begin{algorithm}\label{firstatt_algo}
\begin{enumerate}
\item Enumerate all effective divisors of degree at most the degree of the divisor $D$.  

\item Find an effective divisor $E$ of  smallest degree such that $r(D-E)= -1$ by using Dhar's algorithm. 
\end{enumerate}
\end{algorithm}

\begin{theorem}
The running time of Algorithm \ref{firstatt_algo} is $O(2^{n \ln g})$.
\end{theorem}

The running time of the Algorithm \ref{firstatt_algo} is not polynomial in the size of the input even for a fixed number of vertices since the quantity $2^{n \ln g}$
is not polynomially bounded in the size of the input. There is general interest 
 in obtaining an algorithm that runs in polynomial time for a fixed number of vertices and furthermore, in obtaining a singly exponential time algorithm 
i.e., an algorithm with running time $2^{O(n)}\text{poly}(\text{size}(G))$.  We will now undertake a deeper study of rank to obtain an algorithm that runs in time polynomial in the size of the input provided that the number of vertices is fixed.
More precisely, our algorithm has running time $2^{O(n \log n)} poly(\text{size}(G))$. An important ingredient is a geometric interpretation of rank that we shall obtain in the following section.
%%%%%%%%We start with the formula for rank and then,

\subsection{A geometric interpretation of rank}

We start with the following formula for rank first shown in Baker and Norine \cite{BaNo07} and later reproven in Amini and Manjunath \cite{AmiMan10}.

\begin{theorem}\label{rankform_theo}
For any divisor $D$, we have:
 \begin{equation}r(D)=min_{\nu \in Ext(L_G)}deg^{+}(D-\nu)-1 \end{equation} where  $deg^{+}(D)=\sum_{i:D_i>0}D_i$.
\end{theorem} 

\begin{remark}
Some remarks on the proof(s) of Theorem \ref{rankform_theo} are in order: as we mentioned earlier, Theorem \ref{rankform_theo} has two proofs,
the original proof due to Baker and Norine \cite{BaNo07} was based on combinatorial tools. In particular, the main component of the proof
was to establish the existence and uniqueness of a certain special type of divisors called ``v-reduced'' divisors in each linear equivalence class of divisors,
while the approach of Amini and Manjunath \cite{AmiMan10} involved studying the Laplacian lattice under the simplicial distance function $d_{\triangle}$.
\end{remark}

%%On the other hand, the second proof due to Amini and Manjunath \cite{AmiMan10} involved tools from discrete geometry and lattices, in particular the study of Voronoi diagrams 
%%of the Laplacian lattice under certain simplicial distance functions and the tools involved are closely related to the content of Section \ref{}. 
%%In fact the geometric interpretation of rank that we obtain in Theorem \ref{rankform_theo} was also motivated from the tools involved in the proof.

%%In fact, the formula for rank presented in Theorem \ref{rankform_theo} is a simple consequence of Theorem \ref{structsig_theo}. 

\subsubsection{A sketch of the approach}

Let us now briefly sketch our approach to computing the rank: We start with the formula to compute the rank and proceed as follows: 
we run over all the permutations $\pi \in S_{n+1}$ and for each permutation $\pi$ suppose that we could compute $\min_{q \in L_G}deg^{+}(D-v_{\pi}+q)$ 
in time that is possibly exponential but only in  $n$ then we would obtain an algorithm with running time $O(f(n)poly(\text{size}(G)))$ for some function $f$. 
But, how do we compute $\min_{q \in L_G}deg^{+}(D-v_{\pi}+q)$? 
One hope would be to reduce the problem to a closest vector problem on lattices or more generally to integer programming. 
Fortunately, the integer programming problem has an algorithm that runs in time that 
exponential only in $n$ (the dimension of the lattice). Such an algorithm would run in time $O(2^{n \log n} poly(\text{size}(G)))$. 
This approach requires a better understanding of the $deg^{+}$ function that we now obtain. 

\begin{definition} {\bf (Orthogonal projections onto $H_k$)}
 For a point $P \in \mathbb{R}^{n+1}$ we denote by $\pi_k(P)$ the orthogonal projection of $P$ onto the hyperplane $H_k$.
\end{definition}
For the sake of presentation, we first consider the case where the divisor has degree $g-1$. In this case, we observe that $deg^{+}(D-\nu)=\frac{\ell_1(D-\nu)}{2}$,
and that $\ell_1(D-\nu)=\ell_1(\pi_0(D)- \pi_0(\nu))$.  We denote the set $\pi_0(Ext(L_G))$ the orthogonal projection of $Ext(L_G)$ onto the hyperplane $H_0$ 
by $\text{Crit}_{\triangle}(L_G)$, as defined in Subsection \ref{distfunc_subsect} and indeed, the orthogonal projections of $Ext(L_G)$ are the local maxima of the distance 
function $h_{\triangle,L_G}$ we refer to \cite{AmiMan10} for more details.

\begin{corollary}\label{rankgminusone_cor}
For any divisor $D$ with $deg(D)=g-1$, we have:
 \begin{equation}r(D)=\min_{c \in \text{Crit}_{\triangle}(L_G)}\frac{\ell_1(\pi_0(D)-c)}{2}-1.\end{equation}
\end{corollary}

Taking cue from Corollary \ref{rankgminusone_cor}, it is natural to ask if there is a similar ``distance function'' type interpretation
for divisors of degree between zero and $g-1$. We will answer this question in the affirmative, the relevant distance function, actually a family of 
distance functions is the following:

\begin{definition}({\bf Degree-Plus Distance})
\rm Let $k$ be a positive real number. For points $P$ and $Q$ in $H_0$, we define the generalised degree-plus distance between $P$ and $Q$ as
 $$ d_{k}^{+}(P,Q)= \sup\, \Bigl\{ r\,|\, \triangle(P,r) \cap \triangle(Q,r+k) = \emptyset \Bigr\}.$$
\end{definition}

Note that though $d^{+}_k$ does not appear to be a distance function at first glance, we will actually show that it is can be realised by a sequence of distance functions (See Section \ref{polydistfunc_sect} for a definition)

\begin{figure}
  \begin{center}
    \includegraphics[width=6cm]{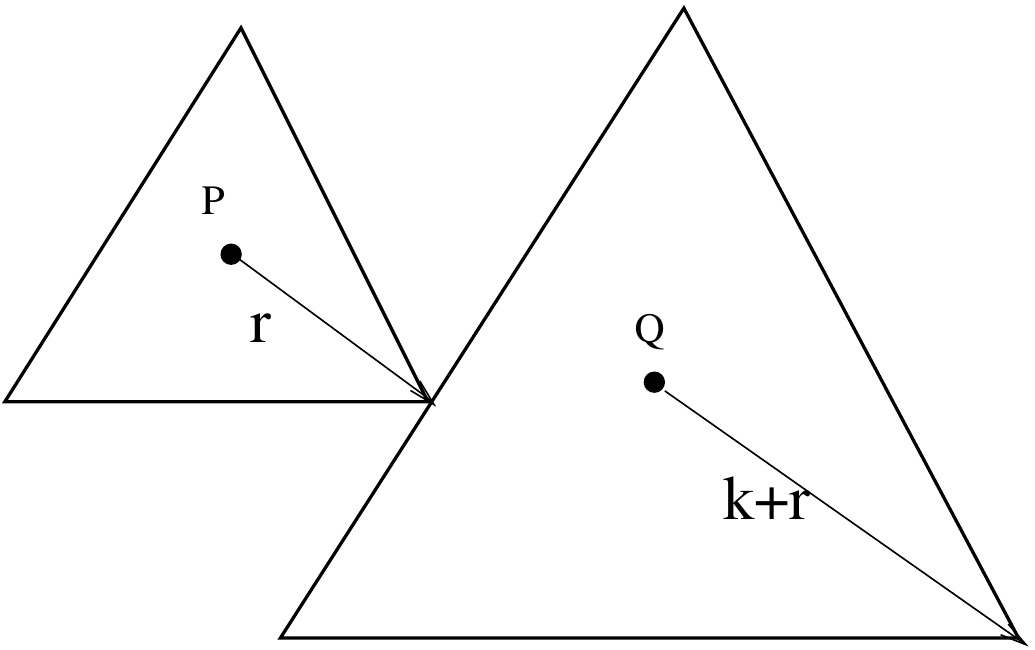} 
  \end{center}
\caption{\label{} The distance function $d^{+}_k$}
 \end{figure}

We will now note some basic properties of the function $d^{+}_k$:

\begin{lemma}({\bf Translation Invariance})\label{trans_lem}
For any points $P$, $Q$, and $T$ in $H_0$ and for any positive real numbers $r_1$ and $r_2$ we have: 
$\triangle(P,r_1) \cap \triangle(Q,r_2)=\emptyset$ if and only if $\triangle(P+T,r_1) \cap \triangle(Q+T,r_2)=\emptyset$.
\end{lemma}

\begin{proof}
Assume that $\triangle(P,r_1) \cap \triangle(Q,r_2) \neq \emptyset$ and consider a point $S \in \triangle(P,r_1) \cap \triangle(Q,r_2)$.
Now, $S= r_1 \sum_{i=1}^{n+1}\alpha_i v_i + P=r_2 \sum_{i=1}^{n+1}\beta_i v_i + Q$ for some $\alpha_i \geq 0$, $\beta_i \geq 0$ and $\sum_i \alpha_i=\sum_i \beta_i=1$. 
Now this implies that $S+T= r_1 \sum_{i=1}^{n+1}\alpha_i v_i + P+T=r_2 \sum_{i=1}^{n+1}\beta_i v_i + Q+T$. Hence, $S+T \in \triangle(P+T,r_1) \cap \triangle(Q+T,r_2) \neq \emptyset$. 
The converse follows by symmetry. %%%%Replace $T$ by $-T$ and use this above lemma.
 \end{proof}

\begin{lemma}({\bf Projection Lemma})\label{proj_lem}
Let $R$ be a point in $\mathbb{R}^{n+1}$ with $deg(R) \geq 0$, let $O$ be the origin and let $\pi_0(R)$ be the orthogonal projection of $R$ onto $H_0$. We have:
\begin{align*}
\inf_{Z \in H^{+}(R)\,\cap\, H^{+}(O)} \, deg(Z)\,\,= (n+1)\,\sup\,\Bigl\{\,r\,|\, \triangle(\pi_0(R),r) \cap \triangle(O,r+\frac{deg(R)}{n+1})=\emptyset\Bigr\}+ deg(R).
\end{align*}
\end{lemma}

\begin{proof}
Consider a point $X$, say, in the intersection of $H^{+}(O)$ and $H^{+}(R)$ and consider the intersection of the hyperplane $H_{deg(X)}$ with $H^{+}(O)$ and $H^{+}(R)$. 
Observe that the intersection of $H_{deg(X)}$ and $H^{+}(O)$ is a simplex that is a scaled and translated copy of $\triangle$, call it $\triangle_1$, centered at $\frac{deg(X)}{n+1}(1,\dots,1)$ and 
scaled by a factor of $\frac{deg(X)}{n+1}$. Similarly, the intersection of $H_{deg(X)}$ and $H^{+}(R)$ is also a simplex that is a scaled and translated copy of $\triangle$, 
call it $\triangle_2$ centered at $R+\frac{deg(X-R)}{n+1}(1,\dots,1)$ scaled by a factor of $\frac{deg(X-R)}{n+1}$. Observe that $deg(X) \geq deg(R) \geq deg(O)$. 
Indeed simplices $\triangle_1$ and $\triangle_2$ intersect at $X$ and $deg(X)$ is equal to $n+1$ times the radius of $\triangle_2$ plus $deg(R)$. We now project 
the simplices $\triangle_1$ and $\triangle_2$ onto $H_0$ and obtain $\inf \{ deg(Z)|$ $Z \in H^{+}(R) \cap H^{+}(O)\} \geq (n+1)\sup\{r|$ $\triangle(\pi_0(R),r) \cap \triangle(O,r+\frac{deg(R)}{n+1})=\emptyset\}+deg(R)$.
Now, consider a point $P$ in $\triangle(\pi_0(R),r) \cap \triangle(O,r+\frac{deg(R)}{n+1})$ and observe that the point $X=P+(r+\frac{deg(R)}{n+1})(1,\dots,1)$ is a point in the 
intersection of $H^{+}(O)$ and $H^{+}(R)$. This shows that $\inf \{ deg(Z)|$ $Z \in H^{+}(R) \cap H^{+}(O)\} \leq (n+1)\sup\{r|$ $\triangle(\pi_0(R),r) \cap \triangle(O,r+\frac{deg(R)}{n+1})=\emptyset\}+deg(R)$. 
This completes the proof.
\end{proof}

We are now ready to establish the connection between the $deg^{+}$ function and the function $d^{+}_k$.

\begin{lemma}\label{dplus_lem}
For any pair of points $P$ and $Q$ in $\mathbb R^{n+1}$ with $deg(P) \geq deg(Q)$, we have 
$$deg^{+}(P-Q)=(n+1)\, d^{+}_{k}(\pi_0(P),\pi_0(Q))+deg(P-Q)$$
 for $k=\frac{deg(P-Q)}{n+1}$
\end{lemma}

\begin{proof}
First consider a point $R$ in $\mathbb{R}^{n+1}$. We have $deg^{+}(R)=\sum_{R_i \geq 0}R_i=deg(R\oplus O)$, where $R \oplus O=(max(R_0,0),\dots,max(R_n,0))$. Now we have:
$$
deg(R\oplus O)= \inf_{Z \in H^{+}(R) \cap H^{+}(O)}  \,deg(Z).
$$
By Lemma \ref{proj_lem} we have
$$\inf_{Z \in H^{+}(R) \cap H^{+}(O)} \, deg(Z) \, = (n+1)\sup\,\Bigl\{\,r\,|\,\triangle(\pi_0(R),r) \cap \triangle(O,r+\frac{deg(R)}{n+1})=\emptyset \Bigr \}+deg(R).$$

Now, for two points $P$ and $Q$ in $\mathbb R^{n+1}$, letting $R=P-Q$ in the above formula and applying Lemma \ref{trans_lem}, we obtain the relation given in the proposition.
%% $\sqrt{\frac{n+1}{n}} \sup\{r|$ $\triangle(\pi_0(P)-\pi_0(Q),r) \cap \triangle(O,r+ \sqrt{\frac{n}{n+1}}deg(P-Q))=\emptyset\}+deg(R)= \\
%%  \sqrt{\frac{n+1}{n}} \sup\{r|$ $\triangle(\pi_0(P),r) \cap \triangle(\pi_0(Q),r+\sqrt{\frac{n}{n+1}}deg(P-Q))=\emptyset\}+deg(P-Q)=\\
%%  \sqrt{\frac{n+1}{n}}d^{+}_k(P,Q)+deg(P-Q)$ for $k=\sqrt{\frac{n}{n+1}}deg(P-Q)$.
%%This completes the proof.
\end{proof}

The function $d^{+}_k$ is motivated naturally by the definition of the $deg^{+}$ function but is not very handy for geometric as well as computational reasons. 
In the following, we will obtain a more convenient representation of $d^{+}_k$. In fact,  $d^{+}_k$ is closely related to the following family of polytopes: 
For a point $P \in H_0$ and $m,n>0$, let $\mathcal{P}_{m,n}(P)=(\triangle(O,m) \oplus_{Mink} \bar{\triangle}(O,n))+P$, where $\oplus_{Mink}$ denotes the Minkowski sum. 
Note that we use the notation $\oplus$ for the tropical maximum sum.

\begin{lemma} For any positive real numbers $m,~n$, $\mathcal{P}_{m,n}(P)$ is a convex polytope.\end{lemma}
\begin{proof}
Using the fact that Minkowski sum of two convex polytopes is a convex polytope and hence, 
$\mathcal{P}_{m,n}(O)$ is a convex polytope. Indeed translates of a convex polytope is also a convex polytope and hence,
$\mathcal{P}_{m,n}(P)$ is also a convex polytope.
\end{proof}

\begin{lemma}\label{degdistfunc_lem}
For points $P$ and $Q$ in $H_0$ and for $k \geq 0$,  $d^{+}_k(P,Q)= \inf\{r|~ Q \in (\triangle(O,r) \oplus_{Mink} \bar{\triangle}(O,r+k))+P\}$, where $O$ is the origin.
\end{lemma}

\begin{proof}
By definition $d^{+}_k(P,Q)=sup\{r|~\triangle(P,r) \cap \bar{\triangle}(Q,r+k)=\emptyset\}$. Let $r_0=d^{+}_k(P,Q)$
and consider a point $R$ in the intersection of $\triangle(P,r_0)$ and $\bar{\triangle}(Q,r_0+k)$. 
Rephrasing  $d_{\triangle}(P,R)=r_0$ and $d_{\triangle}(Q,R)=d_{\bar{\triangle}}(R,Q)=r_0+k$.
This implies that $R-P \in \triangle(O,r_0)$ and $Q-R \in \bar{\triangle}(O,r_0+k)$. 
This shows that $Q-P \in \triangle(O,r_0) \oplus_{Mink} \bar{\triangle}(O,r_0+k)$ and we obtain 
$Q \in (\triangle(O,r_0) \oplus_{Mink} \bar{\triangle}(O,r_0+k))+P$. Hence, $d^{+}(P,Q) \geq \inf\{r|~ Q \in (\triangle(O,r) \oplus_{Mink} \bar{\triangle}(O,r+k))+P\}$.

Furthermore, if $Q$ is contained in  $(\triangle(O,r) \oplus_{Mink} \bar{\triangle}(O,r+k))+P$ 
then there exists a point $R=R_1+R_2$ such that $R_1 \in \triangle(O,r) $
and $R_2 \in  \bar{\triangle}(O,r+k)$ with $Q=R_1+R_2+P$ and we take $R_3$ such that $R_3=Q-R_2=R_1+P$.
Therefore, the point $R_3$ is contained in both  $\triangle(Q,r+k)$ and $\triangle(P,r)$
and we obtain $\inf\{r|~ Q \in (\triangle(O,r) \oplus_{Mink} \bar{\triangle}(O,r+k))+P\} \geq d^{+}_k(P,Q)$.
This concludes the proof.

\end{proof}

As a corollary we obtain a handy characterisation of the polytope $\mathcal{P}_{m,n}$:

\begin{corollary}\label{poly_cor} Let $P,~Q$ be points in $H_0$, a point $Q$ belongs to the  polytope $\mathcal{P}_{r,r+d}(P)$ 
if and only if  $deg^{+}(P+\frac{d(1,\dots,1)}{n+1}-Q) \leq r(n+1)$.\end{corollary}

\begin{proof}
Let a point $Q$ be contained in $\mathcal{P}_{r,r+d}(P)$, then $\inf\{r'|~ Q \in (\triangle(O,r') \oplus_{Mink} \bar{\triangle}(O,r'+d))+P\} \leq r$
and hence by Lemma \ref{degdistfunc_lem} we know that $r \geq \inf\{r'|~ Q \in (\triangle(O,r') \oplus_{Mink} \bar{\triangle}(O,r'+d))+P\}=d^{+}_d(P,Q)$.
By Lemma \ref{dplus_lem}, we know that $deg^{+}(P+\frac{d(1,\dots,1)}{n+1}-Q)=(n+1)d^{+}_{d}(P,Q)$. Hence, $deg^{+}(P+\frac{d(1,\dots,1)}{n+1}-Q) \leq r(n+1)$.
Conversely, if a point $Q$ satisfies $deg^{+}(P+\frac{d(1,\dots,1)}{n+1}-Q) \leq r(n+1)$ then, $d^{+}_{d}(P,Q) \leq r$ and hence, 
$r \geq \inf\{r'|~ Q \in (\triangle(O,r') \oplus_{Mink} \bar{\triangle}(O,r'+d))+P\}$ and hence $Q \in \mathcal{P}_{r,r+d}(P)$.
\end{proof} 

Now putting together, the formula for rank in Theorem \ref{rankform_theo}, Lemma \ref{dplus_lem} and Lemma \ref{degdistfunc_lem} we obtain the following geometric 
interpretation of rank:

\begin{theorem}\label{geominter_theo}({\bf A Geometric Interpretation of rank}) Consider a divisor $D$ of degree $d$ between zero and $g-1$, then $D$ has rank $r_0-1$ if and only if $\pi_0(D)$ is 
contained in the boundary of the arrangement $\cup_{c \in \text{Crit}_{\triangle}(L_G)}\mathcal{P}_{r_1,r_2}(c)$ where $r_1=r_0/(n+1)$ and $r_2=(r_0+g-1-d)/(n+1)$.\end{theorem}

\begin{remark}
The fact that $\mathcal{P}_{r_1,r_1+k}(c) \subseteq \mathcal{P}_{r_2,r_2+k}(c)$ if $r_1 \leq r_2$ is implicit in the statement of Theorem \ref{geominter_theo}.
 \end{remark}

%%%%%%%%%%%\begin{theorem} A separation oracle for the polytopes $\mathcal{P}_{m,n}$ can be constructed in time polynomial in the bitwise description of $m$ and $n$. \end{theorem}

%%%%%%%%%%%For the $\ell_1$-norm of radius $r$ the separating hyperplane for a point $p$ can be constructed by constructed by 
%%%%%%%%%%%%as $\sum_{j=1}^{n} (-1)^{s}x_j=(r'+r)$ where $r'$ is the $\ell_1$-norm of $p$ and ${-1}^{s}$ are the signs of the coordinates of the point $p$.
%%%%%%%%%%%Similar for the polytope $\mathcal{P}_{r,r+k}$ we can construct separating hyperplane by going to their $deg^{+}$ definition.

\subsection{Computing the rank for divisors of degree between zero and $g-1$}

We now give an algorithm for computing the rank that runs in polynomial time for a fixed number of vertices. The algorithm uses two main ingredients: 

\begin{enumerate}

\item The geometric interpretation of rank (Theorem \ref{geominter_theo}).

\item Reduction to the algorithm for integer programming by Kannan \cite{Kan87}:  

\end{enumerate}

For the sake of exposition, we first consider the slightly easier case of divisors with 
degree exactly $g-1$.  We employ Theorem \ref{rankgminusone_cor} to obtain the following algorithm:

%%%%%%%%%%Make it an algorithm.
\begin{algorithm}
\begin{enumerate}
\item For each permutation $\pi \in S_{n+1}$, we compute $\min_{q \in L_G} \ell_1(\pi_0(D)-q)/2-1$ using Kannan's algorithm. The $\ell_1$ unit ball is given to Kannan's algorithm
 as a separation oracle (we will provide an efficient implementation of the separation oracle in Lemma \ref{sepora_lem}).
\item We minimise over all permutations $\pi$.
\end{enumerate}
\end{algorithm}

We now turn to the general case: we start with the geometric interpretation for rank and we would like to reduce the problem to the integer programming problem.
 We construct a preliminary algorithm as follows: 

\begin{algorithm}\label{geom_algo}
\begin{enumerate}
\item Find the smallest integer $r$ such that $\pi_0(D)$ is contained in $\cup_{c \in \text{Crit}_{\triangle}(L_G)}\mathcal{P}_{r_1,r_2}(c)$ where $r_1=\frac{r}{n+1},~r_2=\frac{r+g-1+d}{n+1}$
by testing for all values of $r$ from zero to $g-1$.  

\end{enumerate}
\end{algorithm}

Using the fact that the degree of the divisor is between zero and $g-1$, the algorithm would run in time $O(g\cdot 2^{O(n \log n)} \cdot poly(\text{size}(G)))$.  
Since $g$ is not polynomially bounded in the size of the input, the algorithm does not run in polynomial time for fixed values of $n$. We resolve this 
problem by performing a binary search over the parameter $r$ in the polytope $\mathcal{P}_{r_1,r_2}(c)$ and apply Kannan's algorithm at each step 
of the binary search. Since we know from Subsection \ref{simp_subsect}, 
that the rank of the divisor is at most $g-1$ the algorithm terminates in $O(2^{n \log n} poly(\text{size}(G)))$. Here is a formal description of the algorithm:

\begin{algorithm}\label{geombin_algo}

\begin{enumerate}

\item For each permutation $\pi \in S_{n+1}$, use binary search on the parameter $r$ along with Kannan's algorithm to test if $\pi_0(D)$ is contained 
in  $\cup_{q\in L_G}\mathcal{P}_{r,r+g-1-d}(c_{\pi}+q)$, the polytope is presented to Kannan's algorithm as a separation oracle.
 
\item Repeat over all permutations $\pi$.

\end{enumerate}

\end{algorithm}

The straightforward way of presenting the polytope  $\mathcal{P}_{r,r+k}$ to Kannan's algorithm is  
in terms of its facets. But since the number of facets of $\mathcal{P}_{r,r+k}$ 
is $2^{n+1}$, the factor depending on $n$ in the time complexity of the algorithm becomes larger than $2^{n \log n}$.
Hence, we present the polytope $\mathcal{P}_{r,r+k}$ by a separation oracle to Kannan's algorithm 
and the following efficient implementation of the separation oracle ensures that the 
algorithm runs in time $2^{O(n \log n)}poly(\text{size}(G))$.

\begin{lemma}({\bf A separation oracle for the polytope $\mathcal{P}_{m,n}$})\label{sepora_lem}
There is a polynomial time separation oracle for the polytope $\mathcal{P}_{r,r+d}$ i.e., given any point $p$ there is a polynomial time (in the bit length of $p$ and the vertex 
description of $\mathcal{P}_{r,r+d}$) algorithm that either decides that $p$ is contained in $\mathcal{P}_{r,r+d}$  or outputs a hyperplane separating the point $p$ and the polytope
$\mathcal{P}_{r,r+d}$.
 \end{lemma}
\begin{proof}
Given a point $Q$, compute the function $r'=deg^{+}(P+\frac{d(1,\dots,1)}{n+1}-Q)/(n+1)$ and if $r' \leq r$ then $Q$ is contained in $\mathcal{P}_{r,r+d}(P)$ and
otherwise let $S^{+}$ be the set of indices such that $P_i+d(1,\dots,1)/{n+1}>0$, 
output the hyperplane $H_S$: $\sum_{x_i \in S^{+}}(-x_i+P_i+d(1,\dots,1)/{n+1}) \leq \frac{(r'+r)(n+1)}{2}$ as a separating hyperplane.   To show that $H_S$ is a separating hyperplane, assume the contrary and since $\sum_{x_i \in S^{+}}(-x_i+P_i+\frac{d(1,\dots,1)}{n+1})=r'(n+1)>\frac{(r+r')(n+1)}{2}$ 
there is a point $Q$ in $\mathcal{P}_{m,n}$ such that $\sum_{Q_i \in S^{+}}(P_i+\frac{d(1,\dots,1)}{n+1}-Q_i) \geq \frac{(n+1)(r'+r)}{2}$.
 We know that  for the point $Q^L=Q-\frac{d(1,\dots,1)}{n+1}$ we have:
$\sum_{j \in S^{+}}(Q^L)_j=r'(n+1)>\frac{(r+r')(n+1)}{2}$. Hence, $deg^{+}(Q^L) \geq \frac{(r+r')(n+1)}{2}>r(n+1)$ and $deg(Q^{L})=d$, using Corollary \ref{poly_cor} we obtain a contradiction.

\end{proof}

%%%%%%%%%%%Another approach would be to determine a well-rounded version of the polytope ourselves instead of invoking Loveseat?s rounding algorithmi
%%%%%%%%%%%%%and then invoke Kannan's algorithm.

%%\begin{lemma} For any positive real numbers $R_1,~R_2$ and a point $P \in H_0$, the polytope $\mathcal{P}_{R_1,R_2}(P)$ is approximately well-rounded i.e.,  $ S(P,r_1) \subseteq \mathcal{P}_{m,n}(P) \subseteq S(P,r_2)$ 
%%where $r_1/r_2 \leq n.\sqrt{n}$.\end{lemma}
%%\begin{proof}
%%By definition, the polytope $\mathcal{P}_{R_1,R_2}(P)$ contains the regular simplices $r \triangle(P)$ and $s \bar{\triangle(P)}$. %
%%We will now argue that $\mathcal{P}_{R_1,R_2}(P)$ is contained in $4 r \triangle$ if $r \geq s$ and in  $4 s \bar{\triangle}$
%otherwise. Consider a vertex of $\mathcal{P}_{R_1,R_2}(P)$, by Lemma \ref{vertmin_lem} we know that any vertex is of the form $R_1v_i-R_2v_j$
%% for $i \neq j$, we write $R_1v_i-R_2v_j=(R_1v_i/2+1/2(n R_2\sum_{k \neq j}v_k)/{n})=(-R_2v_j/2+1/2(n R_1\sum_{k \neq j}-v_k)/{n})$ 
%%and hence, $\mathcal{P}_{R_1,R_2}(P)$ is contained in $n\cdots R_1 \triangle$ if $R_1 \geq R_2$ and $n R_2 \bar{\triangle}$ otherwise. 
%%Now we show that the simplices $\triangle$ and $\bar{\triangle}$ are well-rounded.{\color{blue} UNDER CONSTRUCTION.}
%%\end{proof}

The correctness of the algorithm is clear from Theorem \ref{geominter_theo}.

\begin{theorem} For any divisor $D$ with degree between zero to $g-1$, Algorithm \ref{geombin_algo} computes the rank of the divisor $D$.\end{theorem}

\begin{theorem} Algorithm \ref{geom_algo} runs in time $2^{O(n \log n)}poly(\text{size}(G))$ and hence, runs in polynomial time for a fixed number of vertices. \end{theorem}

\begin{proof}
The first step in the algorithm takes time $O(\ln(g) 2^{O(n \log n)})poly(\text{size}(G))$ 
since a separation oracle for $\mathcal{P}_{m,n}$ can be constructed in polynomial time in the size of $G$ and 
Kannan's algorithm takes $2^{O(n \log n)}poly(\text{size}(G))$ and 
we iterate $2^{n \log n}$ times.  Since $\ln(g)$ is polynomially bounded in the size of the input, 
the time complexity of the algorithm is $O( 2^{O(n \log n)} \text{poly}(\text{size}(G))$.
\end{proof}

\subsection{The Algorithm}\label{mainalgo}

We now summarise the results that we obtained in the previous section to obtain an algorithm for computing the rank
of a divisor.

\begin{algorithm}

%%{\bf Input:} The adjacency matrix of a simple graph $G$ and a divisor $D$ (a point in $\mathbb{Z}^{n+1}$ where $n+1$ is the number of vertices of the graph $G$).

\begin{enumerate}
{
\item If $deg(D)<0$ then, output $r(D)=-1$.

\item If $g \leq deg(D) \leq 2g-2$ then, set $D'=K-D$ and compute $r(D')$. Output $r(D)=r(D')+deg(D)-(g-1)$.

\item If $deg(D)>2g-2$, then $r(D)=deg(D)-g$.

\item If $0 \leq deg(D) \leq g-1$, then we invoke Algorithm \ref{geom_algo} to compute $r(D)$.
}
\end{enumerate}

%%{ \bf Output:} The rank of $D$.

\end{algorithm}

\begin{remark}
As described in \cite{AmiMan10}, the notion of rank of a divisor can also be defined for an arbitrary sublattice of the root lattice $A_n$. 
In such a general setting, we do not know if rank can computed in polynomial time even when the dimension of the lattice is fixed. 
In our algorithm we crucially exploit our knowledge of the extremal points and the problem with handling the general case is that we do not have an explicit description of 
the extremal points as we have in the case of Laplacian lattices. As a consequence, we do not know how to find the extremal 
points in polynomial time even when the dimension is fixed. 
\end{remark}

We will end this section by determining the vertices of the polytope $\mathcal{P}_{m,n}$.

\begin{lemma}\label{vertmin_lem} The vertices of $\mathcal{P}_{R_1,R_2}$ are of the form $w_{i,j}=R_1t_i-R_2t_j$ for $i \neq j$ where $t_0,\dots,t_n$ are the vertices of $\triangle$.\end{lemma}

\begin{proof}
Observe that $w_{i,j}$ is contained in $\mathcal{P}_{R_1,R_2}$ for all pairs $i$, $j$ from $0$ to $n$. 
We now show that  $\mathcal{P}_{R_1,R_2}$ is contained in the convex hull of $w_{i,j}$ where $i,~j$ vary from $0$ to $n$. 
Let $p$ be a point in $\mathcal{P}_{R_1,R_2}$ by definition it can written as $R_1\sum^{n}_{i=0}\lambda_iv_i-R_2 \sum_{i=0}^{n}\sigma_iv_i$
with $\lambda_i \geq 0,~\sigma_i \geq 0$ for $i$ from $0$ to $n$ and $\sum_{i=0}^{n} \lambda_i=\sum_{i=0}^{n} \sigma_i=1$.
We let $\ell_{ij}=\lambda_i.\sigma_j$ and write  $p=\sum_{i=0}^{n+}\sum_{j=0}^{n}\ell_{ij}w_{ij}$. We now verify that $\sum_{i,j}\ell_{i,j}=1$ and $\ell_{i,j} \geq 0$.
This shows that $\mathcal{P}_{R_1,R_2}$ is contained in the convex hull of $\{w_{i,j}\}_{i,j}$.

%%By definition, any point $p$ in $\mathcal{P}_{R_1,R_2}$ is of the form $R_1v_i+\bar{F}_i$ where $\bar{F}_i$ is the facet of 
%%$\bar{\triangle}(O,R_2)$ not containing $v_i$ and $-R_2v_i+F_i$ where $F_i$ is the facet of $\triangle(O,R_1)$ not containing $v_i$. 

We now show that $w_{i,i}$ are not vertices of $\mathcal{P}_{R_1,R_2}$ since $w_{i,i}$ is contained in $\triangle(O,R_1-R_2)$ and $\triangle(O,R_1-R_2)$ is contained in $\triangle(O,R_1)$ if $R_1 \geq R_2$ and 
is contained in $\bar{\triangle}(O,R_2)$ otherwise. To conclude the proof of the lemma, it suffices to show that $w_{i,j}$ is a vertex if $i \neq j$. To this end, we consider the linear function $f_{i,j}=x_i-x_j$ and note that $w_{i,j}$ is the unique maximum of 
$f_{i,j}$ among all points in the set $\{w_{i,j}\}_{i,j}$. %%We know that $\mathcal{P}_{R_1,R_2}$ is equal to the convex hull of $\{w_{i,j}\}_{1 \leq i,j \leq n+1}$,
%%the proof is complete.
\end{proof}

%%\begin{lemma} The facets of the polytope  $\mathcal{P}_{r,r+k}(P)$ are of the following form: let $K_S$ be the halfspace $\sum_{j \in S} x_j \leq \sum_{j \in S}p_j+\frac{|S|k}{n+1}+r$
%%the facets are the intersection of $K_S$ and $H_0$ over all non-trivial subsets of $\{1,\dots,n+1\}$ apart from the empty set and $\{1,\dots,n+1\}$. \end{lemma}

%%Proof: For every the facet lifting up gives a hyerplane of the required form and conversely, every hyperplane of the type projected onto $H_0$ must be a 
%%facet of  $\mathcal{P}_{r,r+k}(P)$

%%%%%%%%%%%Hence the number of facets of $K_S$ are subsets of $H_0$.

%%\begin{corollary} The polytope  $\mathcal{P}_{r,r+k}(P)$ has $n(n+1)$ vertices and $2^{n+1}-2$ facets. \end{corollary}

\section{A duality theorem and its generalisations} \label{dual_sect}

We obtain another proof of the fact that testing if a divisor $D$ is effective or not 
i.e., testing if $r(D) \geq 0$ is contained in $NP \cap co-NP$.  Although, this fact 
is already implicit in the algorithm of Tardos, our proof motivates a generalisation of the duality theorem which will see in the next subsection.

An interesting characterisation of divisors of negative rank (i.e., rank equal to minus one) is the following theorem first established in \cite{BaNo07}:

\begin{theorem} {\label{structsig_theo}}
A point $D$ has rank minus one if and only if it dominates a point in $Ext(L_G)$. 
\end{theorem}

Theorem \ref{structsig_theo} was further generalised to arbitrary full dimensional sublattices of $H_0$ in \cite{AmiMan10}.
We will state a  ``projected'' version of the theorem here. Let $L$ be a full-dimensional sublattice of $H_0$.
 For any real number $t>0$, define the arrangements $\mathcal{A}_t=\cup_{q \in L} \bar{\triangle}(q,t)$ and $\mathcal{B}_t=\cup_{c \in \text {Crit}_{\triangle}(L)}\triangle(c,t)$ 
(see Subsection \ref{distfunc_subsect} for the definition). We have the following duality theorem for the arrangements $\mathcal{A}_t$ and $\mathcal{B}_t$.

\begin{theorem}( {\bf A Duality Theorem for the arrangement of Simplices})\label{stansimpdual_theo}
Let $L$ be a full dimensional lattice in $H_0$ for any real number $0 \leq t \leq Cov_{\bar{\triangle}}(L)$,  then the arrangements $\mathcal{A}_{t}$ 
and $\mathcal{B}_{Cov_{\bar{\triangle}}(L)-t}$ tile $H_0$ i.e.,  $int(\mathcal{A}_t) \cup \mathcal{B}_{Cov_{\bar{\triangle}}(L)-t}=H_0$ and 
$int(\mathcal{A}_t) \cap \mathcal{B}_{Cov_{\bar{\triangle}}(L)-t}=\emptyset$. 
\end{theorem}

We now use Theorem \ref{simpdual_theo} and the fact that in the case of Laplacian lattices, the set of $\text{Crit}_{\triangle}(L_G)$ is precisely the projection of $Ext(L_G)$ onto $H_0$ shows that the computational problem of testing whether rank of a divisor is non-negative and is in $NP \cap co-NP$. 
As we shall see the fact that the elements of $\text{Crit}_{\triangle}(L_G)$ have a combinatorial interpretation help to obtain an efficient verification procedure for $r(D)<0$.

%%%%%%%%%Start with the observation that the structural theorem of Sigma-region immediately implies that the problem $r(D) \geq 0$
%%%is in NP intersect co-NP. 

\begin{theorem}\label{efftest_theo}
The problem of deciding if $r(D) \geq 0$ is contained in $NP \cap co-NP$.
\end{theorem}

\begin{proof}
Suppose that $r(D) \geq 0$ then by definition there is a point $q \in L_G$ such that 
$D \geq q$. We can assume that $q$ is contained in $\bar{\triangle}(\pi_0(D),Cov_{\bar{\triangle}}(L_G))$ and hence, 
$q$ has a description that is polynomial in the description of $D$ and $Cov_{\bar{\triangle}}(L_G)$.
We know that $Cov_{\bar{\triangle}}(L_G)$ is the density of the graph (see \cite{AmiMan10}). 
The verifier then tests if $D \geq q$ and whether $q \in L_G$ i.e., 
 and hence, if $q$ can be written as an 
integral combination of the basis elements of $L_G$. This test is the lattice membership problem 
and can be performed in polynomial time, see \cite{Kansur87} for more details. 

This shows that the problem of deciding if $r(D) \geq 0$ is in $NP$. Conversely, to show that the problem of 
testing if $r(D)<0$ is also contained in $NP$, we use the structural theorem of Theorem \ref{structsig_theo}. By Theorem \ref{structsig_theo} 
we know that if $r(D)<0$ then there is a point $c'$ in $\text{Crit}_{\triangle}(L_G)$ of the form $c'=c_{\pi}+q$ 
where $q \in L_G$ such that $c \in \triangle(\pi_0(D),Cov_{\bar{\triangle}}(L_G)-deg(D)/n+1)$. Since, $c'$ is contained in $\triangle(\pi_0(D),Cov_{\bar{\triangle}}(L_G)-deg(D)/n+1)$ 
we know that $c'$ has a description that is polynomial 
in the description of $D$ and the description of $G$.  In order to verify that $c'$ is of the form 
$c'=c_{\pi}+q$ for some $\pi \in S_{n+1}$ and $q \in L_G$, the prover also produces $\pi$ and $q$.
The verifier then tests if $q$ is in $L_G$ by testing if $q$ can be written as an integral combination
of a basis of $L_G$, this test is known as the membership testing problem and can be performed in polynomial time and then test if $c'-q$ is equal to the projection of  $(indeg_{\pi}(v_1)-1,indeg_{\pi}(v_2)-1,\dots,indeg_{\pi}(v_{n+1})-1)$ onto $H_0$. 
Hence, deciding if $r(D)<0$ is also contained in $NP$.
\end{proof}

%%%%%%%%%%%%A simpler way of explaining would be to consider the case where the degree of the divisor is $g-1$
%%%%%%%%%%%%and it is clear that the polytope under question is the $\ell_1$-unit ball and then treat the general case later.

\subsection{Generalised Duality}

Reformulating Corollary \ref{rankgminusone_cor} in a more geometric language, the rank of the divisor is essentially the smallest possible radius of the arrangement of polytopes that are up to scaling the unit ball of the $\ell_1$-norm and centered at points in $\text{Crit}_{\triangle}(L_G)$ that contains $\pi_0(D)$.  
A ``duality theorem'' similar to Theorem \ref{stansimpdual_theo} for the $\ell_1$-unit ball instead of the simplex $\triangle$ will, under some mild assumptions, 
have similar implications as  Theorem \ref{efftest_theo} for the problem of testing if a divisor of degree $g-1$ has rank at most $c$ for any $c$ between $-1$ and $g-1$. 
This observation also raises the problem of extending the notion of duality from the regular simplex $\triangle$ to a general polytope.

\begin{definition}({\bf Duality for polytopes}) \label{dual_defi}
The graph  $G$ is said to have duality with respect to the polytope $\mathcal{P}$ if there exists a star-shaped body  $\mathcal{P}^{*}$ 
such that for $R>0$, there exists an $R'>0$ such that a point $p$ is contained in the interior of the arrangement $\cup_{c \in \text{Crit}_{\triangle}(L_G)}\mathcal{P}(c,R)$ 
if and only if it is not contained in $\cup_{q \in L_G}\mathcal{P}^{*}(q,R')$.
\end{definition}

\begin{remark}
Our imposition that  $\mathcal{P}^{*}$ has to be star-shaped in the definition of duality is somewhat arbitrary and is motivated from the 
fact that in practice $\mathcal{P}^{*}$ happens to be star-shaped.
\end{remark}

While we do not know of tools to obtain duality theorems with respect to polytopes in general we will now sketch a general strategy to prove duality theorems for simplices. Here are the main ingredients:

(I1). From the results in \cite{AmiMan10}, we know that the elements of $\text{Crit}_{\triangle}(L_G)$ have a geometric interpretation namely as  
the local maximum of the simplicial distance function $\triangle$ on the lattice $L_G$. We use the fact that the set of local maxima of the simplicial distance 
function behave ``well'' under non-singular linear maps i.e., if $M$ is a non-singular linear map then $M(\text{Crit}_{\triangle}(L_G)))$ is the set of 
local maxima of $M(L_G)$ under the simplicial distance function $d_{M(\triangle)}$.  

(I2). Suppose that the linear map $M$ fixes $L_G$ and also $\text{Crit}_{\triangle}(L_G)$ as sets then we know that $\text{Crit}_{\triangle}(L_G)$ must 
be the set of local maxima of the simplicial distance function induced by the simplex $M(\triangle)$ on $L_G$. 
This motivates a new notion of  automorphism of a graph i.e., the group of linear transformations that fix both $L_G$ and $\text{Crit}_{\triangle}(L_G)$. 
Furthermore, the graph $G$ has duality with respect to the simplex $M(\triangle)$ with  ${(M(\triangle))}^{*}$ being $-M(\triangle)$.

%% If we have duality with respect to simplices $S_1$ and $S_2$ with centroid at the origin, then we also have duality with respect to the polytopes $S_1 \cap S_2$, with the dual object being $-S_1 \cup -S_2$. {\color{blue} This statement is false and hence there is a basic error in the approach. Hence, the current approach only allows us to study simplices for which the duality theorem holds it is not clear how to generalise the approach for polytopes!} Hence, simplices with respect to which duality holds can be used as a ``basis" to obtain polytopes for which duality by taking intersections and unions. For example, we would be interested in looking at decompositions of a polytope into simplices of the form $M(\triangle)$ where $M$ is an automorphism of $\text{Crit}_{\triangle}(L_G)$. This problem leads to the study of decomposition of a polytope as an intersection of simplices that arise from the automorphisms of $\text{Crit}_{\triangle}(L_G)$ {\color{blue}Make more precise}.

Let us  now discuss each item of the general strategy in more detail.

%%% We start with main structural result that are used to prove the results using in Item (I1). 

%%%%%%%%%%The main facts that we employ are:

%%%%%%%%%%M is a well defined map and hence M(S_1 \cup S_2)=M(S_1) \cap M(S_2).

%%%%%%%%%% Since, M is non-singular hence invertible.

%%%%%%%%%% Since, M is non-singular it is homeomorphism from $H_0$ to itself.

\begin{lemma}{\label{distfunpre_lem}}Let $M$ be a non-singular linear transformation on $H_0$ and let $L$ be a full dimensional lattice in $H_0$, then
for every point $p \in H_0$, $h_{\triangle,L}(p)=h_{M(\triangle),M(L)}(M(p))$. \end{lemma}
\begin{proof}
We know that there exists a lattice point $q \in L$ such that $h_{\triangle,L}(p)=d_{\triangle}(p,q)$. Hence, 
$p \in \triangle(O,h_{\triangle,L}(p))+q$.  We have: $M(p) \in M(\triangle(O,h_{\triangle,L}(p)))+M(q)$ and we have:
$d_{M(\triangle)}(M(p),M(q))=h_{\triangle,L}(p)$ and hence, $h_{M(\triangle),M(L)}(M(p)) \leq h_{\triangle,L}(p)$.
Since, $M$ is non-singular, $M^{-1}$ is well defined and we apply the argument for a point in $M(L)$ and the distance function $d_{M(\triangle)}$ to deduce that $h_{M(\triangle),M(L)}(M(p)) \geq h_{\triangle,L}(p)$.
\end{proof}

Another result we use is the behaviour of arrangements of simplices under non-singular linear transformations.
Let $\mathcal{D}$ be a discrete point set of $H_0$ and let  $\mathcal{A}_{S,\mathcal{D},t}$ be the arrangement of simplices of type $S$ on $\mathcal{D}$, more precisely 
$\mathcal{A}_{\mathcal{D},S,t}=\cup_{q \in \mathcal{D}}S(q,t)$.

\begin{lemma}\label{arrangbeh_lem}
For any non-singular linear transformation $M$ on $H_0$, $M(\mathcal{A}_{\mathcal{D},S,t})=\mathcal{A}_{M(\mathcal{D}),M(S),t}$ . Furthermore, the map $M$ preserves the interior of the arrangements i.e.,
$M(int(\mathcal{A}_{\mathcal{D},S,t}))=int(\mathcal{A}_{(M(\mathcal{D}),M(S),t)})$.
\end{lemma}
\begin{proof}
By definition: $M(\mathcal{A}_{\mathcal{D},S,t})=M(\cup_{q\in \mathcal{D}}S(O,t)+q)=(\cup_{q\in \mathcal{D}}M(S(O,t)+q)=(\cup_{q\in \mathcal{D}}M(S(O,t)))+M(q)=\mathcal{A}_{M(\mathcal{D}),M(S),t}$. Now, suppose $p$ is a point in the interior of the arrangement $\mathcal{A}_{\mathcal{D},S,t}$ then we know that there exists $\epsilon>0$ such that $B(p,\epsilon)$ in the arrangement 
$\mathcal{A}_{\mathcal{D},S,t}$. Applying the linear transformation $M$, we know that $M(B(p,\epsilon))$ is contained in the arrangement $M(\mathcal{A}_{\mathcal{D},S,t})=\mathcal{A}_{(M(\mathcal{D}),M(S),t)}$. Now, we know that $M(B(p,\epsilon))$ contains the ball $B(M(p),\sqrt{\lambda(M)}\epsilon)$ 
where $\lambda(M)=\inf\{ ||M(x)||_2^{2}|~|||x||^{2}=1,~x \in H_0\}$ since $M$ is non-singular we know that $\lambda(M)>0$ and hence, $M(p)$ is contained in the interior of $\mathcal{A}_{(M(\mathcal{D}),M(S),t)}$ and similarly, if we apply the same argument for a point $p'$ in the interior of $\mathcal{A}_{(M(\mathcal{D}),M(S),t)}$ then $M^{-1}(p)$ is a point in the interior of 
$\mathcal{A}_{\mathcal{D},S,t}$.
Hence, $M(int(\mathcal{A}_{\mathcal{D},S,t}))=Int(\mathcal{A}_{(M(\mathcal{D}),M(S),t)})$.
\end{proof}

\begin{corollary}{\label{covradpre_lem}}Let $M$ be a non-singular linear transformation on $H_0$ and let $L$ be a full dimensional lattice in $H_0$, then $Cov_{\triangle}(L)=Cov_{M(\triangle)}(M(L))$.
\end{corollary}

\begin{remark}
The second part of Lemma \ref{arrangbeh_lem} is essentially a topological fact and in topological terms we use the fact a non-singular linear transformation from $H_0$ to itself is a homeomorphism from $H_0$ to itself where $H_0$ is equipped with the Euclidean topology. %%This implies that the arrangements $\mathcal{A}_{\mathcal{D},S,t}$ and $\mathcal{A}_{M(\mathcal{D}),M(S)},t$ equipped with the Euclidean topology are homeomorphic and hence, the interiors of points are preserved. 
\end{remark}

\begin{remark}
Note that we did not use our assumption that the set $\mathcal{D}$ is a discrete point set anywhere in the proof of Lemma \ref{arrangbeh_lem} and can be generalised to non-discrete sets but since we mainly deal with discrete sets we restrict to them.
\end{remark}

\begin{theorem} Let $M$ be a non-singular linear transformation on $H_0$ and let $L$ be a full dimensional lattice in $H_0$, the set of local maxima of the function $h_{M(\triangle),M(L)}$ is precisely the set $M(\text{Crit}_{\triangle}(L))$. \end{theorem}
\begin{proof}
Suppose $c$ is an element of $\text{Crit}_{\triangle}(L)$, we know that there exists an $\epsilon>0$ such that
$h_{\triangle,L}(c) \geq h_{\triangle,L}(c')$ for all $c' \in B(\epsilon,c)$. Now by Lemma \ref{distfunpre_lem},
we know that $h_{M(\triangle),M(L)}(M(c)) \geq h_{M(\triangle),M(L)}(c'')$ for all $c'' \in M(B(\epsilon,c))$.
The set $M(B(\epsilon,c))$ is an ellipsoid that contains the ball $B(M(c),\epsilon \sqrt{\lambda(M)})$ where 
$\lambda(M)$ is $\inf\{ ||M(x)||_2^{2}|~|||x||^{2}_2=1,~x \in H_0\}$ and since $M$ is non-singular we have $\lambda(M)>0$. 
Hence, $M(\text{Crit}_{\triangle}(L)) \subseteq \text{Crit}_{M(\triangle)}(M(L))$. 
Since, $M$ is a non-singular linear map, we consider the map $M^{-1}$ and apply the above argument for a point in $\text{Crit}_{M(\triangle)}(M(L))$ to deduce that 
$\text{Crit}_{M(\triangle)}(M(L)) \subseteq M(\text{Crit}_{\triangle}(L))$.
\end{proof}

%%%%%%Remark that we assume that the linear map and lattice is on $H_0$, since this is what we use, the theorem actually holds for any  subspace of $\mathbb{R}^{n+1}$.

We will now generalise the duality theorem with respect to any full dimensional simplex on $H_0$ with centroid at the origin. 

\begin{theorem} {\label{simpdual_theo}} Let $S$ be a full dimensional simplex on $H_0$ with centroid at the origin and let $\mathcal{A}_{L,S,t}$ be the arrangement 
$\cup_{q \in L}-S(q,t)$ and let $\mathcal{B}_{L,S,t}$ be the arrangement $\cup_{c \in Crit_{S}(L)}S(c,t)$. Fix a real number $0 \leq t \leq Cov_{\triangle}(L)$, 
the arrangements $\mathcal{A}_{L,S,t}$ and $\mathcal{B}_{L,S,Cov_{S}(L)-t}$ tile $H_0$ i.e.,  $\mathcal{A}_{L,S,t} \cup \mathcal{B}_{L,S,Cov_{S}(L)-t}=H_0$
and $int(\mathcal{A}_{L,S,t}) \cap \mathcal{B}_{L,S,Cov_{S}(L)-t}=\emptyset$.\end{theorem}
\begin{proof}
Since $S$ is a full-dimensional simplex with centroid at the origin we know that there is a non-singular map $M$ that takes 
$S$ to the simplex $\triangle$. Furthermore, by Corollary \ref{covradpre_lem} we know that $Cov_{-S}(L)=Cov_{\bar{\triangle}}(M(L))$ and hence $0 \leq t \leq Cov_{\bar{\triangle}}(M(L))$ and we can apply Theorem \ref{stansimpdual_theo} to the lattice $M(L)$, a full dimensional lattice in $H_0$ and with the simplex $\triangle$, that a point $p \in H_0$ is contained exclusively in either exclusively the interior of the arrangement $\mathcal{A}_{M(L),\triangle,t}$ or the arrangement $\mathcal{B}_{M(L),\triangle,Cov_{\bar{\triangle}}(M(L))-t}$. Now note that $M(L)$ and $\text{Crit}_{\triangle}(M(L))$ are discrete sets and apply Lemma \ref{arrangbeh_lem} with the linear transformation $M^{-1}$. 
Hence, we know that $int(\mathcal{A}_{L,S,t}) \cup \mathcal{B}_{L,S,Cov_{S}(L)-t}=H_0$
and $int(\mathcal{A}_{L,S,t}) \cap \mathcal{B}_{L,S,Cov_{S}(L)-t}=\emptyset$ since, if there is a point $p \in H_0$
in $int(\mathcal{A}_{L,S,t}) \cap \mathcal{B}_{L,S,Cov_{S}(L)-t}$ then $M^{-1}(p)$ is contained in 
$int(\mathcal{A}_{M(L),\triangle,t}) \cap \mathcal{B}_{(M(L),\triangle,Cov_{\triangle}(M(L))-t)}$
 which is a contradiction.\end{proof}

\subsection{The Automorphism Group of $\text{Crit}_{\mathcal{P}}(L)$}

%%%%%%%%%i. Give the general definition for a general full-dimensional polyhedron $\mathcal{P}$ and $L$.
We now discuss the second item of the general strategy, the notion of automorphisms:

\begin{definition} ({\bf Automorphism of a lattice})
For a lattice $L$, a non-singular linear map $M$ on the linear space spanned by the elements of $L$ is called an automorphism of $L$ if it induces a bijection on $L$.
\end{definition}

Here is a concrete characterisation of the automorphism group of a lattice:

\begin{lemma}
Let $\mathcal{B}$ be an arbitrary basis of a $n$-dimensional lattice $L$ written as a matrix with the basis elements row-wise. A linear transformation $M$ is an automorphism of the  lattice $L$ if and only if it $M(\mathcal{B})=T \cdot \mathcal{B}$ where $T \in SL(n,\mathbb{Z})$.
\end{lemma}

The set of automorphisms of a lattice has a natural group structure under composition (if you think of the linear maps as matrices then the group operation is matrix multiplication) and this group is called the automorphism group of $L$.

\begin{definition} ({\bf Automorphism of $\text{Crit}_{\triangle}(L)$})
Let $L$ be a sublattice of $H_0$, a non-singular linear map $M$ on the linear space spanned by the elements of $L$ is called an automorphism of $\text{Crit}_{\triangle}(L)$ if it induces a bijection on $\text{Crit}_{\triangle}(L)$. 
\end{definition}

The set of automorphisms of $\text{Crit}_{\triangle}(L)$ has a natural group structure under composition (if you think of the linear maps as matrices then the group operation is matrix multiplication). The group gives rise to a graph invariant that we call the {\bf critical automorphism group} of the graph.
%%%%%%%%%%

\begin{remark}
Note that the notion of a critical automorphism group can be developed more generally, for any distance function induced by a convex body $\mathcal{P}$ on a lattice $L$, we can define the group of transformations that fix both $L$ and $\text{Crit}_{P}(L)$ as sets.
\end{remark}

A convenient way to think of the automorphisms of $\text{Crit}_{\triangle}(L_G)$ is to consider $\text{Crit}_{\triangle}(L_G)/L_G$ as a subset of the torus $\mathbf{T}_G=H_0/L_G$ and observe the automorphisms of $L_G$ induce an bijection of the torus $\mathbf{T}_G$ to itself; among these bijections the automorphisms of $\text{Crit}_{\triangle}(L_G)$ are those that take the set $\text{Crit}_{\triangle}(L_G)/L_G$ to itself. Let us now give some concrete examples of $\text{Crit}_{\triangle}(L_G)$:\\

\begin{enumerate}

\item For a regular graph $G$, we have $\text{Crit}_{\triangle}(L_G)=-\text{Crit}_{\triangle}(L_G)$. Hence, the linear transformation $M(x)=-x$ is an example 
of an automorphism of $\text{Crit}_{\triangle}(L_G)$.

\item The permutation map induced by every automorphism of $H_0$ also is indeed an automorphism of $\text{Crit}_{\triangle}(L_G)$; 
but since permutation maps fix the simplex $\triangle$, these are not useful to obtain simplices other than $\triangle$ for which duality holds. 
Note that these automorphism take one class in $\text{Crit}_{\triangle}(L_G)$ to another.

\item Every $c \in \text{Crit}_{\triangle}(L_G)$ where $c=c_{\pi}+q$ and $q \in L_G$. We can write each $c_{\pi}$ as a rational linear combination of $\{b_1,\dots,b_{n}\}$. 
Define the height of $c_{\pi}$ as the least common multiple of the denominators of these rational coefficients in their reduced form and define the height of $\text{Crit}_{\triangle}(L_G)$ 
as least common multiple of the ratio of the denominators of the height of each $c_{\pi}$. A linear transformation $M(b_i)=b_i+H(\sum_{j<i}\alpha_jb_j$) is an automorphism of $L_G$ 
where $H$ is a multiple of the height of $\text{Crit}_{\triangle}(L_G)$ and $\alpha_i$s are arbitrary integers is an automorphism of $\text{Crit}_{\triangle}(L_G)/L_G$. 
These automorphisms takes an equivalence class $\text{Crit}_{\triangle}(L_G)/L_G$ to itself.

\end{enumerate}

The usefulness of the critical automorphism group arises from the following partial characterization of simplices (with centroid at the origin) for which duality holds: 
\begin{theorem}\label{dual_theo}
Let $G$ be an undirected connected graph, duality holds with respect to a simplex $S$ with centroid at the origin if
there is a linear map $M$ taking $\triangle$ to $S$ that  belongs to the automorphism group of $\text{Crit}_{\triangle}(L_G)$.
\end{theorem}
%%\begin{proof}
%%We already showed that if $M$ is an element of critical automorphism group then  $M(\triangle)$ has duality. Conversely, if $S$ has duality then $\text{Crit}_{\triangle}(L_G)$ 
%%is the precisely the set of local maxima under the distance function $d_S$ (prove this!). Now, consider a map $M$ that takes $S$ to $\triangle$ we know that $M(Crit_{\triangle})(L_G)$ 
%%is the set of local maxima of $d_{\triangle}$ and hence, $M(Crit_{\triangle})(L_G)=Crit_{\triangle}(L_G)$.
%%\end{proof}
\begin{remark}
We do not know if the converse of Theorem \ref{dual_theo} also holds, the main difficulty in proving the converse is to show that if $S$ has duality then $\text{Crit}_{\triangle}(L_G)$ 
is the precisely the set of local maxima under the distance function $d_S$. 
\end{remark}

The fact that an automorphism $M$ of $\text{Crit}_{\triangle}(L_G)$ fixes $\text{Crit}_{\triangle}(L_G)$ 
combined with the Theorem \ref{simpdual_theo} means that the graph $G$ has a duality with respect to the simplex $M(\triangle)$
with the dual $(M(\triangle))^{*}$ being $-M(\triangle)$. 

%%\begin{lemma}{\label{dualinter_lem}}{\color{blue} This statement is false!}
%%If a graph has duality with respect to simplices $S_1,\dots,S_m$ then it also has duality with respect to the polytope $S_1 \cap S_2 \cap \dots S_m$ with the dual object being $-S_1 \cup -S_2 \cup -S_3 \cdots -S_m$ and also under $S_1 \cup S_2 \cdots S_m$ with the dual object being $-S_1 \cap -S_2 \cdots \cap -S_m$.
%%\end{lemma}

%%%%%%Give the generalised duality and how it leads to the construction of new graph invariants.

%%%%%%%%%iv. Appendix: A note on the sensitivity of the problem.

\subsection{The case of complete graphs}

%%%%%%%%%%TO DO: Work out some explicit examples of automorphism groups.

We will study the critical automorphism group of the complete graph. 
%%%%%%%%%%%Make a table???
Let $b_0,\dots,b_{n}$ be the rows of the Laplacian matrix of the complete graph $K_{n+1}$.
The extremal points are permutations of the point $(-1,0,1,\dots,n)$ of the coordinates translated by a lattice
point i.e., a point in the Laplacian lattice of $K_{n+1}$. Since, the points in $\text{Crit}_{\triangle}(L_G)$
are orthogonal projections of the extremal points onto the hyperplane $H_0=(1,\dots,1)^{\perp}$, 
they are points of the form $c_{\pi}+q$ where $c_{\pi}=\pi((n/2,n/2-1,n/2-2,\dots,-n/2))$, 
where $\pi$ is a permutation matrix. Now, since we now that up to equivalence modulo $L_G$
there are only $n!$ class of $\text{Crit}_{\triangle}(L_G)$ and that permutations such that $\pi(n)=n$
form representatives of the set $\text{Crit}_{\triangle}(L_G)/L_G$. We now consider permutations 
such that $\pi(n)=n$ and for such a permutation, we can rewrite $c_{\pi}$ in the basis of $\{b_0,\dots,b_{n}\}$
as 

\begin{equation} \label{crit_eq}
c_{\pi}=\frac{(\sum_{j=0}^{n-1}(j+1) \cdot b_\pi(j))}{n}.
 \end{equation}

Hence, the height of $\text{Crit}_{\triangle}(K_{n+1})$ is $n$. We can use the exploit
Equation (\ref{crit_eq}) to obtain a better formula for rank. We will now construct critical automorphisms of $K_{n+1}$:

\begin{lemma}
Every map of the form $b_{i}=b_{\pi(i)}$ where $\pi$ is a permutation is an element of the critical automorphism group.
\end{lemma}

More generally, we have the following result:

\begin{theorem}\label{compcrit_theo}
A map $M$ that takes $b_{i}$ to $b_{\pi(i)}+H \cdot q_i$ for $i$ from $0$ to $n-1$ where $\pi$ is a permutation and $H$ is an integer that divides $n$,
 the height of $K_{n+1}$ and $q_i$ is contained in the sublattice spanned by $b_{\pi(0)},\dots,b_{\pi(i-1)}$ i.e., $q_i=\sum_{j=0}^{i-1}\alpha_{ij}b_{\pi(j)}$
where $\alpha_{ij}$s are integers is an element of the critical automorphism group of $K_{n+1}$.
\end{theorem}

\begin{proof}
First, we show that the map takes $L_{K_{n+1}}$ to itself as follows:
we write $b_{\pi(i)}$ for $i$ from $0$ to $n$ as an integral combination 
of $M(b_{\pi(0)}),\dots,M(b_{\pi(n)})$ as follows:
$b_{\pi(0)}=M(b_i)$ and for $i$ from $1$ to $n-1$ we have:\\

\begin{equation}
b_{\pi(i)}=M(b_i)-H \sum_{j=0}^{i-1} \sum_{k=0}^{i-1}\alpha_{jk} M(b_j)
\end{equation}

Hence, $M$ takes a basis of $L_{K_{n+1}}$ to another basis of $L_{K_{n+1}}$ and hence takes $L_{K_{n+1}}$ to itself.

Now, to show that $M$ takes $\text{Crit}_{\triangle}(L_{K_{n+1}})$ to itself we consider $M(c_{\pi})$
and form Equation (\ref{crit_eq}) we have:

$M(c_{\sigma})=\frac{(\sum_{j=0}^{n-1}(j+1) \cdot M(b_\pi(j)))}{n}=\frac{(\sum_{j=0}^{n-1} (j+1) \cdot b_{\pi(\sigma(j)}+H \cdot q_{\pi(j)}}{n})$.
Now since $H$ divides $n$ we  know that $M(c_{\sigma})$ is mapped to $c_{\sigma \circ \pi}+q'$ for some lattice point $q'$ and since $M$ takes $K_{n+1}$ to itself
it takes $M(c_{\sigma}+L_{K_{n+1}})$ to $c_{\sigma \circ \pi}+L_{K_{n+1}}$. Furthermore, 
since $\pi$ and $\sigma$ are permutations, $M$ induces a permutation on equivalence classes in $\text{Crit}_{\triangle}(L_{K_{n+1}})/L_{K_{n+1}}$ represented by $c_{\pi}$ with $\pi(n)=n$.
These arguments show $M$ that takes $\text{Crit}_{\triangle}(L_{K_{n+1}})$ to itself and hence $M$ is an element of the critical automorphism group of $K_{n+1}$. 
\end{proof}

\begin{remark}
Note that Theorem \ref{compcrit_theo} provides a construction of elements of the critical automorphism group of $K_{n+1}$, but the construction is not exhaustive in the sense,
there are other elements of the critical automorphism group of $K_{n+1}$ that are not constructed by Theorem \ref{compcrit_theo}. 
\end{remark}

%%%%%%%%Verify that the map induces a bijection on $Crit_{\triangle}(L_G)$ to itself.

\subsection{Concluding remarks}

A natural outgrowth of our work would be to obtain a singly exponential algorithm for computing the rank and furthermore, to resolve the complexity of computing the rank i.e., to determine 
whether it is NP-hard or it is computatble in polynomial time. The bottleneck in our algorithm is the enumeration of all permutations. 
 We believe that another interesting direction would be obtain duality theorems for polytopes generalising the results that we obtained on simplices.

{\bf Acknowledgements:} The author thanks Omid Amini, Khaled Elbasionni, Jan van den Heuvel and Amr Elmasary for the stimulating discussions they 
had with him on the topic of the paper.

}
\end{document}